\documentclass[a4paper, 12point]{article}
\usepackage{}
\newcommand{\be}{\begin{equation}}
\newcommand{\ee}{\end{equation}}
\newcommand{\bea}{\begin{eqnarray}}
\newcommand{\eea}{\end{eqnarray}}

\usepackage{graphicx,epsfig}
\usepackage{enumerate}
\usepackage{amssymb}
\usepackage{mathrsfs}
\usepackage{amsthm}
\usepackage{amsmath}
\usepackage{setspace}
\usepackage{endnotes}
\newtheorem{thm}{Theorem}[section]
\theoremstyle{definition}

\theoremstyle{lemma}
\newtheorem{lem}{Lemma}[section]
\theoremstyle{example}

\newtheorem{rmrk}{Remark}[section]
\theoremstyle{illustration}

\theoremstyle{proposition}

\theoremstyle{corollary}
\newtheorem{cor}{Corollary}[section]
\numberwithin{equation}{section}
\begin{document}
\date{}
\title{\textbf{Some $m$th-order Difference Sequence Spaces of Generalized Means and Compact Operators}}
\author{Amit Maji\footnote{Corresponding author, e-mail:
amit.iitm07@gmail.com}, Atanu Manna \footnote{Author's e-mail: atanumanna@maths.iitkgp.ernet.in}, P. D. Srivastava \footnote{Author's e-mail: pds@maths.iitkgp.ernet.in}\\
\textit{\small{Department of Mathematics, Indian Institute of Technology, Kharagpur}} \\
\textit{\small{Kharagpur 721 302, West Bengal, India}}}
\maketitle
\vspace{20pt}
\begin{center}\textbf{Abstract}\end{center}
In this paper, new sequence spaces $X(r, s, t ;\Delta^{(m)})$ for $X\in \{l_\infty, c, c_0\}$ defined by using generalized means and difference operator of order $m$ are introduced. It is shown that these spaces are complete normed linear spaces and the spaces $c_0(r, s, t ;\Delta^{(m)})$, $c(r, s, t ;\Delta^{(m)})$ have Schauder basis. Furthermore, the $\alpha$-, $\beta$-, $\gamma$- duals of these spaces are computed and also obtained necessary and sufficient conditions for some matrix transformations from $X(r, s, t ;\Delta^{(m)})$ to $X$. Finally, some classes of compact operators on the spaces $c_0(r, s, t ;\Delta^{(m)})$ and $l_{\infty}(r, s, t ;\Delta^{(m)})$ are characterized by using the Hausdorff measure of noncompactness. 

\section{Introduction}
The study of sequence spaces has importance in the several branches of analysis, namely, the structural theory of topological vector spaces, summability theory, Schauder basis theory etc. Besides this, the theory of sequence spaces is a powerful tool for obtaining some topological and geometrical results using Schauder basis.

Let $w$ be the space of all real or complex sequences $x=(x_n)$, $n\in \mathbb{N}_0$. For an infinite matrix $A$ and a sequence space $\lambda$, the matrix domain of $A$, which is denoted by $\lambda_{A}$ and defined as $\lambda_A=\{x\in w: Ax\in \lambda\}$ \cite{WIL}. Basic methods, which are used to determine the topologies, matrix transformations and inclusion relations on sequence spaces can also be applied to study the matrix domain $\lambda_A$. In recent times, there is an approach of forming new sequence spaces by using matrix domain of a suitable matrix and characterize the matrix mappings between these sequence spaces.

Kizmaz first introduced and studied the difference sequence space in \cite{KIZ}. Later on,
several authors including Ahmad and Mursaleen \cite{AHM}, \c{C}olak
and Et \cite{COL}, Ba\c{s}ar and Altay \cite{ALT}, Polat and Ba\c{s}ar \cite{POL1}, Aydin and Ba\c{s}ar \cite{AYD} etc. have introduced and studied new sequence spaces defined by using difference operator.

On the other hand, sequence spaces are also defined by using
generalized weighted mean. Some of them can be viewed in Malkowsky
and Sava\c{s} \cite{MAL}, Altay and Ba\c{s}ar \cite{ALT1}. Mursaleen and Noman \cite{MUR} also introduced a sequence
space of generalized means, which includes most of the earlier known
sequence spaces. But till $2011$, there was
no such literature available in which a sequence space is generated by
combining both the weighted mean and the difference operator. This
was first initiated by Polat et al. \cite{POL}. Later on, Ba\c{s}arir et al. \cite{BAS} generalized the sequence spaces of Polat et al. \cite{POL} to an
$m$th-order difference sequence spaces $X(u, v; \Delta^{(m)})$ for $X \in \{ l_{\infty}, c, c_{0} \}$ which is defined as
\begin{center}
$X(u, v; \Delta^{(m)})= \Big\{x =(x_n)\in w : \big((G(u, v).\Delta^{(m)}x)_n\big) \in X \Big\}$,
\end{center}where $u, v\in w$ such that $u_n, v_n \neq 0$ for all $n$, $\Delta^{(m)} = \Delta^{(m-1)} \circ \Delta^{(1)}$ for $m \in \mathbb{N}$ and the matrices $G(u, v)=(g_{nk})$, $\Delta^{(1)}=(\delta_{nk})$ are defined by
\begin{align*}
g_{nk} &= \left\{
\begin{array}{ll}
    u_nv_k & \quad \mbox{~if~} 0\leq k \leq n,\\
    0 & \quad \mbox{~if~} k > n
\end{array}\right.&
\delta_{nk}& = \left\{
\begin{array}{ll}
    0 & \quad \mbox{~if~} 0\leq k <n-1 \\
    (-1)^{n-k} & \quad \mbox{~if~} n-1\leq k \leq n,\\
    0 & \quad \mbox{~if~} k>n.
\end{array}\right.
\end{align*}
respectively.\\

The aim of this present paper is to introduce new sequence spaces
defined by using both the generalized means and the difference
operator of order $m$. We investigate some topological properties as well as the $\alpha$-, $\beta$-, $\gamma$- duals and bases of the new
sequence spaces are obtained. We also characterize some matrix
mappings between these new sequence spaces. Finally, we give the characterization of some classes of compact operators on the spaces $c_0(r, s, t ;\Delta^{(m)})$ and $l_{\infty}(r, s, t ;\Delta^{(m)})$ by using the Hausdorff measure of noncompactness.
\section{Preliminaries}
Let $l_\infty, c$ and $c_0$ be the spaces of all bounded, convergent and null sequences $x=(x_n)$ respectively, with the norm $\|x\|_\infty=\displaystyle\sup_{n}|x_n|$. Let  $bs$ and $cs$ be the sequence spaces of all bounded and convergent series respectively. We denote by $e=(1, 1, \cdots)$ and $e_{n}$ for the sequence whose $n$-th term is $1$ and others are zero and $\mathbb{{N_{\rm 0}}}=\mathbb{N}\cup \{0\}$, where $\mathbb{N}$ is the set of all natural numbers.
A sequence $(b_n)$ in a normed linear space $(X,
\|.\|)$ is called a Schauder basis for $X$ if for every $x\in X$
there is a unique sequence of scalars $(\mu_n)$ such that
\begin{center}
$\Big\|x-\displaystyle\sum_{n=0}^{k}\mu_nb_n\Big\|\rightarrow0$ as $k\rightarrow\infty$,
\end{center}
i.e., $x=\displaystyle\sum_{n=0}^{\infty}\mu_nb_n$ \cite{WIL}.\\
For any subsets $U$ and $V$ of $w$, the multiplier space $M(U, V)$ of $U$ and $V$ is defined as
\begin{center}
$M(U, V)=\{a=(a_n)\in w : au=(a_nu_n)\in V ~\mbox{for all}~ u\in U\}$.
\end{center}
In particular,
\begin{center}
$U^\alpha= M(U, l_1)$, $U^\beta= M(U, cs)$ and $U^\gamma= M(U, bs)$
\end{center} are called the $\alpha$-, $\beta$- and $\gamma$- duals of $U$ respectively \cite{MAL1}.

Let $A=(a_{nk})_{n, k}$ be an infinite matrix
with real or complex entries $a_{nk}$. We write $A_n$ as the
sequence of the $n$-th row of $A$, i.e.,
$A_n=(a_{nk})_{k}$ for every $n$.
For $x=(x_n)\in w$, the $A$-transform of $x$ is defined as the
sequence $Ax=((Ax)_n)$, where
\begin{center}
$A_n(x)=(Ax)_n=\displaystyle\sum_{k=0}^{\infty}a_{nk}x_k$,
\end{center}
provided the series on the right side converges for each $n$. For any two sequence spaces $U$ and $V$, we denote by $(U, V)$, the class of all infinite matrices $A$ that map from $U$ into $V$. Therefore $A\in (U, V)$ if and only if $Ax=((Ax)_n)\in V$ for all $x\in U$. In other words, $A\in (U, V)$ if and only if $A_n \in U^\beta$ for all $n$ \cite{WIL}.\\ \\
The theory of $BK$ spaces is the most powerful tool in the characterization of matrix transformations between sequence spaces. A sequence space $X$ is called $BK$ space if it is a Banach space with continuous coordinates $p_n: X\rightarrow \mathbb{K}$, where $\mathbb{K}$ denotes the real or complex field and $p_n(x)=x_n$ for all $x=(x_n)\in X$ and each $n\in \mathbb{N}_{0}$.
The space $l_1$ is a $BK$ space with the usual norm defined by $\|x\|_{l_1}=\displaystyle\sum_{k=0}^{\infty}|x_k|$.
An infinite matrix $T=(t_{nk})_{n,k}$ is called a triangle if $t_{nn}\neq 0$ and $t_{nk}=0$ for all $k>n$.
Let $T$ be a triangle and $X$ be a $BK$ space. Then $X_T$ is also a $BK$ space with the norm given by $\|x\|_{X_T}= \|Tx\|_X$ for all $x\in X_T$ \cite{WIL}.

\section{Sequence spaces $X(r, s, t; \Delta^{(m)})$ for $X \in \{ l_{\infty}, c, c_{0}\}$}
In this section, we first begin with the notion of generalized means given by Mursaleen et al. \cite{MUR}.\\
We denote the sets $\mathcal{U}$ and $\mathcal{U}_{0}$ as
\begin{center}
$ \mathcal{U} = \Big \{ u =(u_{n}) \in w: u_{n} \neq
0~~ {\rm for~ all}~~ n \Big \}$ and $ \mathcal{U_{\rm 0}} = \Big \{ u
=(u_{n}) \in w: u_{0} \neq 0 \Big \}.$
\end{center}
Let $r, t \in \mathcal{U}$ and $s \in \mathcal{U}_{0}$. The sequence $y=(y_{n})$ of generalized means of a sequence $x=(x_{n})$ is defined
by $$ y_{n}= \frac{1}{r_{n}}\sum_{k=0}^{n} s_{n-k}t_{k}x_{k} \qquad (n \in \mathbb{N_{\rm 0}}).$$
The infinite matrix $A(r, s, t)$ of generalized means is defined by

$$(A(r,s,t))_{nk} = \left\{
\begin{array}{ll}
    \frac{s_{n-k}t_{k}}{r_{n}} & \quad 0\leq k \leq n,\\
    0 & \quad k > n.
\end{array}\right. $$

Since $A(r, s, t)$ is a triangle, it has a unique inverse and the
inverse is also a triangle \cite{JAR}. Take $D_{0}^{(s)} =
\frac{1}{s_{0}}$ and

$ D_{n}^{(s)} =
\frac{1}{s_{0}^{n+1}} \left|
\begin{matrix}
    s_{1} & s_{0} &  0 & 0 \cdots & 0 \\
    s_{2} & s_{1} & s_{0}& 0 \cdots & 0 \\
    \vdots & \vdots & \vdots & \vdots    \\
    s_{n-1} & s_{n-2} & s_{n-3}& s_{n-4} \cdots & s_0 \\
      s_{n} & s_{n-1} & s_{n-2}& s_{n-3} \cdots & s_1
\end{matrix} \right| \qquad \mbox{for}~ n =1, 2, 3, \cdots $\\ \\

Then the inverse of $A(r, s, t)$ is the triangle $B= (b_{nk})_{n, k}$, which is defined as
$$b_{nk} = \left\{
\begin{array}{ll}
    (-1)^{n-k}~\frac{D_{n-k}^{(s)}}{t_{n}}r_{k} & \quad 0\leq k \leq n,\\
    0 & \quad k > n.
\end{array}\right. $$
We now introduce the sequence spaces $X(r, s, t; \Delta^{(m)})$ for $X \in \{ l_{\infty}, c, c_{0} \}$ as
$$ X(r, s,t; \Delta^{(m)})= \Big \{ x=(x_{n})\in w : ( (A(r,s,t). \Delta^{(m)}) {x})_{n}) \in X \Big  \},$$
which is a combination of the generalized means and the difference operator of order $m$. By using matrix domain, we can write $X(r, s,t; \Delta^{(m)})=  X_{A(r, s,t; \Delta^{(m)})}=\{x \in w :A(r, s,t; \Delta^{(m)})x\in X\}$, where $A(r, s,t; \Delta^{(m)})= A(r, s,t). \Delta^{(m)}$, product of two triangles $A(r, s,t)$ and $\Delta^{(m)}$. The sequence $y=(y_n)$ is $A(r,s,t). \Delta^{(m)}$-transform of a sequence $x=(x_n)$, i.e.,
$$y_n = \displaystyle \frac{1}{r_n}\displaystyle\sum_{j=0}^{n} \bigg[\displaystyle\sum_{i=j}^{n}(-1)^{i-j} \binom{m}{i-j}s_{n-i} t_i\bigg]x_j.$$
These sequence spaces include
many known sequence spaces studied by several authors. For examples,

\begin{enumerate}[I.]
\item if $r_{n}=\frac{1}{u_{n}}$, $t_{n}=v_{n}$, $s_{n}=1$ $\forall n$, then the sequence spaces $ X(r, s,t; \Delta^{(m)})$ for
$X \in \{ l_{\infty}, c, c_{0} \}$ reduce to $ X(u, v; \Delta^{(m)})$ studied by Ba\c{s}arir et al.\cite{BAS} and in particular for $m=1$, the sequence spaces $ X(u, v; \Delta)$ introduced by Polat et al. \cite{POL}.
\item if $r_{n}= \frac{1}{n!}, $ $t_{n}=\frac{\alpha^n}{n!}$, $s_{n}=\frac{(1-\alpha)^n}{n!}$, where $0<\alpha<1$, then the sequence spaces $ X(r, s,t; \Delta^{(m)})$ for
$X \in \{l_\infty, c, c_0 \}$ reduce to $e^\alpha_\infty(\Delta^{(m)})$,
$e^\alpha(\Delta^{(m)})$ and $e^\alpha_0(\Delta^{(m)})$ respectively studied by Polat and Ba\c{s}ar \cite{POL1}.
\item if $r_{n}=n+1,$ $t_{n}={1+ \alpha^n}$, where $0<\alpha<1$ and $s_{n}=1~\forall n$, then the sequence spaces $ X(r, s,t; \Delta^{(m)})$ for $X \in \{c, c_0 \}$
 reduce to the spaces of sequences $a_{c}^{\alpha}(\Delta)$ and $a_{0}^{\alpha}(\Delta)$ studied by Aydin and Ba\c{s}ar \cite{AYD}. For $X=l_\infty$, the sequence space $ X(r, s,t; \Delta^{(m)})$ reduces to $a_{\infty}^{\alpha}(\Delta)$ studied by Djolovi\'{c} \cite{DJO}.
 \item if $r_n = {\lambda_n}$ $t_n = \lambda_n - \lambda_{n-1}$, $s_n =1$ and $m=1$ then the spaces $X(r, s, t; \Delta^{(m)})$ for  $X \in \{c, c_0\}$
 reduce to $c_0^{\lambda}(\Delta )$ and $c^{\lambda}(\Delta )$ respectively studied by Mursaleen and Noman \cite{MUR4}.
\end{enumerate}
\section{Main results}
In this section, we begin with some topological results of the
newly defined sequence spaces.
\begin{thm}
The sequence spaces $X(r,s, t; \Delta^{(m)})$ for $X\in \{l_\infty, c, c_0
\}$ are complete normed linear spaces under the norm defined by
\begin{center}
$\|x\|_{X(r, s, t; \Delta^{(m)})}=\displaystyle\sup_n\bigg|\frac{1}{r_n}\displaystyle\sum_{j=0}^{n} \bigg[\displaystyle\sum_{i=j}^{n}(-1)^{i-j} \binom{m}{i-j}s_{n-i} t_i\bigg]x_j\bigg|=\displaystyle\sup_n|(A(r, s, t; \Delta^{(m)})x)_n| $
\end{center}
\end{thm}
\begin{proof}
Since $\Delta^{(m)}$ is a linear operator, it is easy to show that $X(r, s, t; \Delta^{(m)})$ is a linear space and the functional $\|.\|_{X(r, s, t; \Delta^{(m)})}$ defined above gives a norm on the linear space $X(r, s, t; \Delta^{(m)})$.\\
To show completeness, let $(x^i)$ be a Cauchy sequence in $X(r, s, t; \Delta^{(m)})$, where $x^{i}= (x_k^{i})=(x_0^{i}, x_1^{i}, x_2^{i}, \ldots )$ $\in X(r, s, t; \Delta^{(m)})$ for each $i\in \mathbb{N}_0$. Then for every $\epsilon>0$ there exists $i_0\in \mathbb{N}$ such that
\begin{center}
$\|x^{i}-x^{j}\|_{X(r, s, t; \Delta^{(m)})}<\epsilon$ ~ for $i, j\geq i_0$.
\end{center}
The above implies that for each $k\in \mathbb{N_{\rm 0}}$,
\begin{equation}
|(A(r, s, t). \Delta^{(m)})(x_k^{i}-x_k^{j})|<\epsilon  ~~\mbox{for all}~ i, j\geq i_0,
\end{equation}
Therefore the sequence $((A(r,s,t).\Delta^{(m)})x_k^{i})_i$ is a Cauchy
sequence of scalars for each $k\in \mathbb{N_{\rm 0}}$ and hence $((A(r,s,t).\Delta^{(m)})x_k^{i})_i$ converges for each $k$. We write
\begin{center}
$\displaystyle\lim_{i\rightarrow\infty}(A(r,s,t).\Delta^{(m)})x_k^{i} =
(A(r,s,t).\Delta^{(m)})x_k$ \quad for each $k \in \mathbb{N}_{0}.$
\end{center}
Letting $j\rightarrow\infty$ in $(4.1)$, we obtain
\begin{equation}
\Big|(A(r, s, t). \Delta^{(m)})(x_k^{i}-x_k)\Big|<\epsilon  ~\mbox{for
all}~ i \geq i_0 ~\mbox{and each}~ k\in \mathbb{N_{\rm
0}}.
\end{equation}
Hence by definition, $\|x^{i}-x\|_{X(r, s, t; \Delta^{(m)})}<\epsilon$
for all $i \geq i_0$. Next we show that $x \in X(r,s,t;
\Delta^{(m)})$. Since $(x^{i}) \in X(r,s, t; \Delta^{(m)})$, we have
\begin{center}
$\|x\|_{X(r, s, t; \Delta^{m})}\leq \|x^{i}\|_{X(r, s, t; \Delta^{(m)})} + \|x^{i}-x\|_{X(r, s, t; \Delta^{(m)})}$,
\end{center}
which is finite for $i \geq i_0$.
So $x\in X(r, s, t; \Delta^{(m)})$. This completes the proof.
\end{proof}

\begin{thm}
The sequence spaces $X(r,s,t; \Delta^{(m)})$ for $X \in \{ l_{\infty}, c, c_{0} \}$ are linearly isomorphic to the spaces $X \in \{ l_{\infty}, c, c_{0} \}$
respectively, i.e., $l_{\infty}(r, s, t; \Delta^{(m)}) \cong l_{\infty}$,  $c(r, s, t; \Delta^{(m)}) \cong c$ and $c_{0}(r, s, t; \Delta^{(m)}) \cong c_{0}$.
\end{thm}
\begin{proof}
We prove the theorem only for the case $X=c_0$. For this, we need to show that there exists a bijective linear map from $c_0(r,s,t;\Delta^{(m)})$ to $c_0$.\\
 We define a map $T:c_0(r,s,t;\Delta^{(m)})\rightarrow c_0$ by $x\longmapsto Tx=y=(y_n)$, where
$$y_n= \frac{1}{r_n}\displaystyle\sum_{j=0}^{n} \bigg[\displaystyle\sum_{i=j}^{n}(-1)^{i-j} \binom{m}{i-j}s_{n-i} t_i\bigg]x_j.$$

Since $\Delta^{(m)}$ is a linear operator, so the linearity of $T$ is trivial. It is clear from the definition
that $Tx=0$ implies $x=0$. Thus $T$ is injective. To prove $T$ is surjective, let $y =(y_n)\in c_0$.
Since $y= (A(r, s, t). \Delta^{(m)})x$, i.e., $$x= (A(r, s, t). \Delta^{(m)})^{-1}y=({\Delta^{(m)}})^{-1}.A(r, s, t)^{-1}y.$$
So we can get a sequence $x=(x_{n})$ as
\begin{equation}
x_n = \sum_{j=0}^{n}\sum_{k=j}^{n} (-1)^{k-j} \binom{m+n-k-1}{n-k}\frac{D_{k-j}^{(s)}}{t_{k}}r_{j}y_{j},  \qquad n \in \mathbb{N_{\rm 0}}.
\end{equation}
Then
$$\|x \|_{c_0(r,s,t; \Delta^{(m)})}= \sup_{n }\bigg| \frac{1}{r_n}\displaystyle\sum_{j=0}^{n} \bigg[\displaystyle\sum_{i=j}^{n}(-1)^{i-j} \binom{m}{i-j}s_{k-i} t_i\bigg]x_j\bigg| = \sup_{n}|y_{n}|= \|y \|_{{\infty}} < \infty.$$
Thus $x \in c_0(r, s, t; \Delta^{(m)})$ and this shows that $T$ is
surjective. Hence $T$ is a linear bijection from $c_0(r,s,t;
\Delta^{(m)})$ to $c_0$. Also $T$ is norm preserving. So $c_0(r, s, t; \Delta^{(m)}) \cong c_0$.\\
Similarly, we can prove that $l_{\infty}(r, s, t; \Delta^{(m)}) \cong l_{\infty}$, $c(r, s, t; \Delta^{(m)}) \cong c$. This completes the proof.
\end{proof}

Since $X(r,s,t; \Delta^{(m)}) \cong X$ for $X \in \{c_{0}, c\}$, the Schauder bases of the sequence spaces $X(r,s,t; \Delta^{(m)})$ are the inverse image of the bases of $X$ for $X \in \{c_{0}, c\}$. So, we have the following theorem without proof.
\begin{thm}
Let $\mu_k=(A(r,s, t; \Delta^{(m)})x)_k$, $k\in \mathbb{N_{\rm
0}}$. For each $j \in \mathbb{N_{\rm 0}}$, define the sequence
$b^{(j)}=(b_{n}^{(j)})_{n}$ of the elements of
the space $c_0(r, s, t; \Delta^{(m)})$ as
$$
b_n^{(j)} = \left\{
\begin{array}{ll}
 \displaystyle\sum_{k=j}^{n} (-1)^{k-j} \binom{m+n-k-1}{n-k}\frac{D_{k-j}^{(s)}}{t_{k}}r_{j}& \mbox{if}~~   0\leq j\leq n \\
0 & \mbox{if} ~~ j>n
\end{array}
\right.
$$ and
$$
b_n^{(-1)} =\displaystyle\sum_{j=0}^{n}\displaystyle\sum_{k=j}^{n} (-1)^{k-j} \binom{m+n-k-1}{n-k}\frac{D_{k-j}^{(s)}}{t_{k}}r_{j}
.$$

Then the followings are true:\\
 $(i)$ The sequence $(b^{(j)})_{j=0}^{\infty}$ is a basis for the space $c_0(r, s,t; \Delta^{(m)})$ and any $x\in c_0(r, s,t; \Delta^{(m)})$ has a unique representation of the form
\begin{center}
$x=\displaystyle\sum_{j=0}^{\infty}\mu_jb^{(j)}$.
\end{center}
$(ii)$ The set $(b^{(j)})_{j=-1}^{\infty}$ is a basis for the space $c(r,
s,t; \Delta^{(m)})$ and any $x\in c(r, s,t; \Delta^{(m)})$ has a unique
representation of the form
\begin{center}
$x=\ell b^{(-1)}+ \displaystyle\sum_{j=0}^{\infty}(\mu_j-\ell)b^{(j)}$,
\end{center}where $\ell=\displaystyle\lim_{n\rightarrow\infty}(A(r,s, t; \Delta^{(m)})x)_n$.
\end{thm}
\begin{rmrk}
In particular, if we choose $r_{n}=\frac{1}{u_{n}}$,
$t_{n}=v_{n}$, $s_{n}=1$ $\forall ~n$, then the sequence spaces $ X(r,
s,t; \Delta^{(m)})$ reduce to $ X(u, v; \Delta^{(m)})$ for $X \in \{c_{0}, c \}$. With this choice of $s_{n}$,
we have $D_{0}^{(s)} =D_{1}^{(s)}=1$ and $D_{n}^{(s)} =0 $ for $ n
\geq 2$. Then the sequences $b^{(j)} = ( b_{n}^{(j)})$ for $j=-1, 0, 1, \ldots$ reduce to
$$
b_n^{(j)} = \left\{
\begin{array}{ll}
 \displaystyle\sum_{k=j}^{j+1} (-1)^{k-j} \binom{m+n-k-1}{n-k}\frac{1}{u_{j}v_{k}}& \mbox{if}~~   0\leq j\leq n \\
0 & \mbox{if} ~~ j>n.
\end{array}
\right.
$$ and
$$
b_n^{(-1)} =\displaystyle\sum_{j=0}^{n}\displaystyle\sum_{k=j}^{j+1} (-1)^{k-j} \binom{m+n-k-1}{n-k}\frac{1}{u_{j}v_{k}}.
$$
 The sequences $( b^{(j)})_{j=0}^{\infty}$ and $(b^{(j)})_{j=-1}^{\infty}$ are the bases for the spaces $c_{0}(u,v; \Delta^{(m)})$ and $c(u,v; \Delta^{(m)})$ respectively $\cite{BAS}$.
\end{rmrk}

Let $\mathcal{F}$ be the collection of all finite nonempty subsets of the set of all natural numbers.
Let $A=(a_{nk})_{n,k}$ be an infinite matrix and consider the following conditions:\\
\begin{align}
&\sup_{K \in \mathcal{F}} \sum_{n}\Big|\sum_{k \in K}a_{nk}\Big| < \infty\\
&\displaystyle\sup_{n} \sum_{k=0}^{\infty}|a_{nk}| < \infty\\
&\displaystyle\lim_{n} \sum_{k=0}^{\infty}|a_{nk}| =0\\
&\displaystyle\lim_{n} a_{nk} =0 \mbox{~for all~} k\\
&\displaystyle\lim_{n} \sum_{k=0}^{\infty}a_{nk}=0\\
&\displaystyle\lim_{n} a_{nk} \mbox{~exists for all~} k\\
&\displaystyle\lim_{n} \sum_{k=0}^{\infty}|a_{nk}- \lim_n a_{nk}| =0.\\
&\displaystyle\lim_{n} \sum_{k=0}^{\infty}a_{nk} \mbox{~exists~}
\end{align}

We now state some results given by Stieglitz and Tietz \cite{STI} which are required to obtain the duals and matrix transformations.
\begin{thm} \cite{STI}
$(a)$ $A \in (c_0, l_1), A \in (c,l_1), A \in (l_{\infty}, l_1)$ if and only if $(4.4)$ holds.\\
$(b)$ $A \in (c_{0}, l_{\infty}), A \in (c, l_{\infty}), A \in (l_{\infty}, l_{\infty})$ if and only if $(4.5)$ holds.\\
$(c)$ $A \in (c_0, c_0)$ if and only if $(4.5)$ and $(4.7)$ hold.\\
$(d)$ $A \in (l_{\infty}, c_0)$ if and only if $(4.6)$ holds.\\
$(e)$ $A \in (c,c_0)$ if and only if $(4.5)$, $(4.7)$ and $(4.8)$ hold.\\
$(f)$ $A \in (c_0, c)$ if and only if $(4.5)$ and $(4.9)$ hold.\\
$(g)$ $A \in (l_{\infty}, c)$ if and only if $(4.5)$, $(4.9)$ and $(4.10)$ hold.\\
$(h)$ $A \in (c, c)$ if and only if $(4.5)$, $(4.9)$ and $(4.11)$ hold.
\end{thm}

\subsection{The $\alpha$-, $\gamma$-duals of $X(r,s, t; \Delta^{(m)})$ for $X\in\{l_\infty, c, c_0 \}$}
Now we compute the $\alpha$-, $\gamma$-duals of $X(r,s, t; \Delta^{(m)})$ for $X\in\{l_\infty, c, c_0 \}$.
\begin{thm}
The $\alpha$-dual of the space $X(r,s,t; \Delta^{(m)})$ for $X\in\{l_\infty, c, c_0\}$ is the set
$$ \Lambda = \Big \{ a=(a_{n}) \in w: \sup_{K \in \mathcal{F}} \sum_{n}\Big|\sum_{j \in K}\sum_{k=j}^{n} (-1)^{k-j} \binom{m+n-k-1}{n-k}\frac{D_{k-j}^{(s)}}{t_{k}}r_{j}a_n\Big| < \infty \Big \}.$$
\end{thm}
\begin{proof}
Let $a=(a_{n}) \in w$, $x\in X (r, s, t; \Delta^{(m)})$ and $y\in X$ for $X \in \{ l_{\infty}, c, c_{0}\}$. Then for each $n \in \mathbb{N_{\rm 0}}$, we have
$$  a_{n}x_n = \sum_{j=0}^{n}\sum_{k=j}^{n} (-1)^{k-j} \binom{m+n-k-1}{n-k}\frac{D_{k-j}^{(s)}}{t_{k}}r_{j}a_ny_{j} =(Cy)_{n},$$
where the matrix $C=(c_{nj})_{n, j}$ is defined as
$$
c_{nj} = \left\{
\begin{array}{ll}
 \displaystyle\sum_{k=j}^{n} (-1)^{k-j} \binom{m+n-k-1}{n-k}\frac{D_{k-j}^{(s)}}{t_{k}}r_{j}a_n& \mbox{if}~~   0\leq j \leq n \\
0 & \mbox{if} ~~ j>n
\end{array}
\right.
$$
and $x_n$ is given by $(4.3)$.
Thus for each $x \in X(r,s,t; \Delta^{(m)})$, $(a_nx_{n})_{n} \in l_{1}$ if
and only if $(Cy)_{n} \in l_{1}$, where $y \in X $ for $X  \in
\{l_{\infty}, c, c_{0} \}$. Therefore $a=(a_{n}) \in [X(r,s,t;
\Delta^{(m)})]^{\alpha}$ if and only if $C \in (X, l_1)$. By using Theorem 4.4(a), we have
$$ [X(r,s,t; \Delta^{(m)})]^{\alpha} = \Lambda.$$
\end{proof}

\begin{thm}
The $\gamma$-dual of the space $X(r, s, t;\Delta^{(m)})$ for $X\in\{l_\infty, c, c_0\}$ is the set $$\Gamma = \Big\{a=(a_n)\in w: ~~ \displaystyle\sup_{l}\displaystyle\sum_{n=0}^{\infty}|e_{ln}|<\infty\Big\},$$
where the matrix $E=(e_{ln})$ is defined by
\begin{equation}
e_{ln}= \left\{
\begin{array}{ll}
  \displaystyle r_n \bigg[ \frac{a_{n}}{s_0t_n} + \sum_{k=n}^{n+1}(-1)^{k-n} \frac{D_{k-n}^{(s)}}{t_k} \sum_{j=n+1}^{l}\binom{m+j-k-1}{j-k}a_j+ \\
~~~~~~~~~~~~~~~~~~~~\displaystyle\sum_{k=n+2}^{l}(-1)^{k-n} \frac{D_{k-n}^{(s)}}{t_{k}}\sum_{j=k}^{l}\binom{m+j-k-1}{j-k}a_{j} \bigg] &
~0\leq n \leq l,\\
    0 & ~ n > l.
\end{array}\right.
\end{equation}
Note: We mean $\sum\limits_{j =n}^{l} =0$ if $n > l$.
\end{thm}

\begin{proof}
Let $a =(a_n)\in w$, $x \in X(r,s,t; \Delta^{(m)})$ and $y \in X$ for $X\in\{l_\infty, c, c_0\}$, which are connected by the relation (4.3). Then, we have

\begin{align*}
\displaystyle\sum_{n=0}^{l}a_n x_n &= \sum_{n=0}^{l}\sum_{j=0}^{n}\sum_{k=j}^{n} (-1)^{k-j} \binom{m+n-k-1}{n-k}\frac{D_{k-j}^{(s)}}{t_{k}}r_{j}a_ny_{j}\\
& =\sum_{n=0}^{l-1}\sum_{j=0}^{n}\sum_{k=j}^{n}(-1)^{k-j}\binom {m+n-k-1} {n-k}\frac{D_{k-j}^{(s)}}{t_{k}}{r_{j}y_{j}a_{n}} + \sum_{j=0}^{l}\sum_{k=j}^{l}(-1)^{k-j}\binom {m+l-k-1} {l-k}\frac{D_{k-j}^{(s)}}{t_{k}}{r_{j}y_{j}a_{l}}\\
\end{align*}
\begin{align*}
& =  \bigg[\frac{D_{0}^{(s)}}{t_{0}}a_0 + \sum_{k=0}^{1}(-1)^k \frac{D_k^{(s)}}{t_k} \sum_{j=1}^{l}\binom{m+j-k-1}{j-k}a_j + \sum_{k=2}^{l}(-1)^{k} \frac{D_{k}^{(s)}}{t_{k}} \sum_{j=k}^{l}\binom{m+j-k-1}{j-k}a_j  \bigg]r_0 y_0 \\ &  +
\bigg[\frac{D_{0}^{(s)}}{t_{1}}a_1 + \sum_{k=1}^{2}(-1)^{k-1} \frac{D_{k-1}^{(s)}}{t_k} \sum_{j=2}^{l}\binom{m+j-k-1}{j-k}a_j + \sum_{k=3}^{l}(-1)^{k-1} \frac{D_{k-1}^{(s)}}{t_{k}} \sum_{j=k}^{l} \binom{m+j-k-1}{j-k}a_j \bigg]r_1 y_1\\
& + \cdots + \frac{D_{0}^{(s)}}{t_l}a_{l}r_ly_{l}\\
& = \sum_{n=0}^{l}r_n \bigg[ \frac{a_{n}}{s_0t_n} + \sum_{k=n}^{n+1}(-1)^{k-n} \frac{D_{k-n}^{(s)}}{t_k} \sum_{j=n+1}^{l}\binom{m+j-k-1}{j-k}a_j+ \\
& ~~~~~~~~~~~~~~~~~~~~~~~~~~~~~~~~~~~~~~~~~~~\sum_{k=n+2}^{l}(-1)^{k-n} \frac{D_{k-n}^{(s)}}{t_{k}}\sum_{j=k}^{l}\binom{m+j-k-1}{j-k}a_{j} \bigg]y_{n} \\
&= (Ey)_{l},
\end{align*}
where $E$ is the matrix defined in $(4.12)$. \\Thus $a \in \big[X(r,s,t;
\Delta^{(m)})\big]^{\gamma}$ if and only if $ax=(a_nx_n)\in bs$ for
$x\in X(r,s,t; \Delta^{(m)})$ if and only if
$\Big(\displaystyle\sum_{n=0}^{l}a_n x_n \Big) \in
l_{\infty}$, i.e., $(Ey)_{l} \in l_{\infty}$ for $y\in X$. Hence by Theorem 4.4(b), we have
 $$ \big[X(r,s,t; \Delta^{(m)})\big]^{\gamma} = \Gamma.$$
\end{proof}
\begin{rmrk}
In particular, if we choose $r_{n}=\frac{1}{u_{n}}$,
$t_{n}=v_{n}$, $s_{n}=1$ $\forall ~n$, then the sequence spaces $ X(r,
s,t; \Delta^{(m)})$ for $X \in \{l_\infty, c, c_{0} \}$ reduce to $ X(u, v; \Delta^{(m)})$ \cite{BAS}. With this choice of $s_{n}$,
we have $D_{0}^{(s)} =D_{1}^{(s)}=1$ and $D_{n}^{(s)} =0 $ for $ n
\geq 2$.
Therefore the $\gamma$-dual of the space $X(u, v;\Delta^{(m)})$ for $X\in\{l_\infty, c, c_0\}$ is the set $$\Big\{a=(a_n)\in w: ~~ \displaystyle\sup_{l}\displaystyle\sum_{n=0}^{\infty}\bigg| \frac{1}{u_n} \bigg[ \frac{a_{n}}{v_n} + \sum_{k=n}^{n+1} \frac{(-1)^{k-n}}{v_k} \sum_{j=n+1}^{l}\binom{m+j-k-1}{j-k}a_j \bigg] \bigg|<\infty\Big\}.$$
\end{rmrk}
\subsection{$\beta$-dual and Matrix transformations}
Here we first discuss about the $\beta$-dual and then
characterize the matrix transformations. Let $T$ be a triangle and
$X_{T}$ be the matrix domain of $T$ in $X$.
\begin{thm}\label{1}
{(\cite{JAR}, Theorem 2.6)} Let $X $ be a BK space with AK property and
$R=S^{t}$, the transpose of $S$, where $S=(s_{jk})$ is the inverse of the matrix $T$. Then
$a \in (X_{T})^{\beta} $ if and only if
 $a \in (X^{\beta})_{R}$ and $W \in (X, c_{0})$, where the triangle $W=(w_{pk})$ is defined by $w_{pk}  = \sum\limits_{j=p}^{\infty}a_{j}s_{jk}$.
 Moreover if $a \in (X_{T})^{\beta}$, then
$$ \sum\limits_{k=0}^{\infty}a_{k}z_{k} = \sum\limits_{k=0}^{\infty} R_{k}(a)T_{k}(z)  \qquad \forall ~z \in X_{T}.$$
\end{thm}
\begin{rmrk}{(\cite{JAR}, Remark 2.7)}\label{r_1}
The conclusion of the Theorem $\ref{1}$ is also true for $X=l_\infty$.
\end{rmrk}

\begin{rmrk}(\cite{MAL1}, \cite{JAR}) \label{r_2}
We have $a \in (c_{T})^{\beta}$ if and only if $R(a) \in l_{1}$ and $W \in (c,c)$.
Moreover, if $a \in (c_{T})^{\beta}$ then we have for all $z \in c_{T}$
$$\sum\limits_{k=0}^{\infty}a_{k}z_{k} = \sum\limits_{k=0}^{\infty}R_{k}(a)T_{k}(z) - \eta \gamma, $$
where $\eta = \displaystyle \lim_{k\rightarrow \infty} T_{k}(z)$ and $\gamma = \displaystyle \lim_{p \rightarrow \infty} \sum\limits_{k=0}^{p}w_{pk}$.
\end{rmrk}

To find the $\beta$-duals of the sequence spaces $X(r, s, t; \Delta^{(m)})$ for $X\in \{l_{\infty}, c, c_0\}$, we define the following sets:
\begin{align*}
 &B_1 = \Big \{ a \in w : \sum\limits_{k=0}^{\infty} |R_{k}(a)| < \infty \Big \}\\
 &B_2 = \Big \{ a \in w : \lim_{p \rightarrow \infty}w_{pk} =0 ~~{\rm for~ all}~ k \Big \}\\
 &B_3 = \Big \{ a \in w : \sup_{p}\sum\limits_{k=0}^{\infty}|w_{pk}| < \infty \Big \}\\
 &B_4 = \Big \{ a \in w :  \lim_{p \rightarrow \infty} \sum\limits_{k=0}^{p}|w_{pk}|=0 \Big \}\\
 &B_5 = \Big \{ a \in w :  \lim_{p \rightarrow \infty}w_{pk} ~~ {\rm~ exists~ for~ all }~ k \Big \}\\
 &B_6 = \Big \{ a \in w :  \lim_{p \rightarrow \infty}\sum\limits_{k=0}^{p}w_{pk} {\rm~~ exists}\Big \},
\end{align*}
where $ R_{k}(a)
   =  r_{k} \bigg [ \frac{a_{k}}{s_{0}t_{k}} + \sum\limits_{i=k}^{k+1}(-1)^{i-k} \frac{D_{i-k}^{(s)}}{t_{i}} \sum\limits_{j=k+1}^{\infty}\binom{m+j-i-1} {j-i}  a_{j}  +
\sum\limits_{l=2}^{\infty} (-1)^{l} \frac{D_{l}^{(s)}}{t_{l+k}}\sum\limits_{j=k+l}^{\infty} \binom{m+j-k-l-1} {j-k-l} a_{j} \bigg ] $ and \\
$w_{pk} = r_{k} \bigg [\sum\limits_{i=k}^{p}(-1)^{i-k} \frac{D_{i-k}^{(s)}}{t_{i}}\sum\limits_{j=p}^{\infty}\binom{m+j-i-1} {j-i} a_{j} +
\sum\limits_{i=p+1}^{\infty}(-1)^{i-k} \frac{D_{i-k}^{(s)}}{t_{i}}\sum\limits_{j=i}^{\infty}\binom{m+j-i-1} {j-i}a_j \bigg ].$

\begin{thm}
We have $[c_{0}(r, s, t; \Delta^{(m)})]^\beta = B_{1} \bigcap B_{2}\bigcap B_{3}$, $[l_{\infty}(r, s, t; \Delta^{(m)})]^\beta=B_{1} \bigcap B_{4} $ and \\$[c(r, s, t; \Delta^{(m)})]^\beta=B_{1} \bigcap B_{3}\bigcap B_{5} \bigcap B_{6}$.
\end{thm}
\begin{proof}
Here the triangle $T=A(r, s, t).\Delta^{(m)}$. So $T^{-1}= (A(r, s, t). \Delta^{(m)})^{-1}=(\Delta^{(m)})^{-1}.{A(r, s, t)}^{-1}$. Let $S=(s_{jk})$ be the inverse of $T$. Then we have
\begin{displaymath}
  s_{jk}  = \left\{
     \begin{array}{ll}
        {  \sum\limits_{i=k}^{j}(-1)^{i-k} \binom{m+j-i-1} {j-i}\frac{D_{i-k}^{(s)}}{t_{i}} r_{k}}  &  \mbox{if}~~ \quad 0\leq k \leq j\\
        0              & \mbox{if}~~ \quad  k > j.
     \end{array}
   \right.
\end{displaymath}
To find the $\beta$-dual of $X(r, s, t; \Delta^{(m)})$ for $X\in \{l_{\infty}, c, c_0\}$, we need to show $R(a)=(R_k(a))\in l_1$, where $R = S^t$ and characterize the classes $W\in(c_0, c_0), W\in(l_{\infty}, c_0)$ and $W\in(c, c)$. Now
\begin{align*}
   R_{k}(a)  & = \sum\limits_{j =k}^{\infty} a_{j}s_{jk}\\
   & = \sum\limits_{j =k}^{\infty} \sum\limits_{i=k}^{j}(-1)^{i-k} \binom{m+j-i-1} {j-i}\frac{D_{i-k}^{(s)}}{t_{i}} r_{k}a_{j} \\
  & = \frac{D_{0}^{(s)}}{t_{k}} r_{k}a_k  +  \sum\limits_{j =k+1}^{\infty} \sum\limits_{i=k}^{j}(-1)^{i-k} \binom{m+j-i-1} {j-i}\frac{D_{i-k}^{(s)}}{t_{i}} r_{k}a_{j}\\
  & = \frac{D_{0}^{(s)}}{t_{k}} r_{k}a_k  +  \sum\limits_{i=k}^{k+1}(-1)^{i-k} \binom{m+k-i} {k-i+1}\frac{D_{i-k}^{(s)}}{t_{i}} r_{k}a_{k+1} + \sum\limits_{i=k}^{k+2}(-1)^{i-k}\binom{m+k-i+1} {k-i+2}\frac{D_{i-k}^{(s)}}{t_{i}} r_{k}a_{k+2} + \cdots \\
  & =   r_{k} \bigg [ \frac{a_{k}}{s_{0}t_{k}} + \sum\limits_{i=k}^{k+1}(-1)^{i-k} \frac{D_{i-k}^{(s)}}{t_{i}} \sum\limits_{j=k+1}^{\infty}\binom{m+j-i-1} {j-i}  a_{j} +
\sum\limits_{l=2}^{\infty} (-1)^{l} \frac{D_{l}^{(s)}}{t_{l+k}}\sum\limits_{j=k+l}^{\infty} \binom{m+j-k-l-1} {j-k-l} a_{j} \bigg ]
\end{align*}and
\begin{align*}
w_{pk} & = \sum\limits_{j=p}^{\infty}a_{j}s_{jk} \\
& = \sum\limits_{j=p}^{\infty}\sum\limits_{i=k}^{j}(-1)^{i-k} \binom{m+j-i-1} {j-i}\frac{D_{i-k}^{(s)}}{t_{i}} r_{k}a_j\\
& = r_{k} \bigg [\sum\limits_{i=k}^{p}(-1)^{i-k} \binom{m+p-i-1} {p-i}\frac{D_{i-k}^{(s)}}{t_{i}} a_{p} +
\sum\limits_{j=p+1}^{\infty}\sum\limits_{i=k}^{j}(-1)^{i-k} \binom{m+j-i-1} {j-i}\frac{D_{i-k}^{(s)}}{t_{i}}a_j \bigg ]\\
&= r_{k} \bigg [\sum\limits_{i=k}^{p}(-1)^{i-k} \frac{D_{i-k}^{(s)}}{t_{i}}\sum\limits_{j=p}^{\infty}\binom{m+j-i-1} {j-i} a_{j} +
\sum\limits_{i=p+1}^{\infty}(-1)^{i-k} \frac{D_{i-k}^{(s)}}{t_{i}}\sum\limits_{j=i}^{\infty}\binom{m+j-i-1} {j-i}a_j \bigg ].
\end{align*}
Using Theorem \ref{1} and Remark \ref{r_1} \& \ref{r_2}, we have
$[c_{0}(r, s, t; \Delta^{(m)})]^\beta = B_{1} \bigcap B_{2}\bigcap B_{3}$, $[l_{\infty}(r, s, t; \Delta^{(m)})]^\beta=B_{1} \bigcap B_{4} $ and $[c(r, s, t;\Delta^{(m)})]^\beta=B_{1} \bigcap B_{3}\bigcap B_{5} \bigcap B_{6}$.
\end{proof}

\begin{thm} {(\cite{JAR}, Theorem 2.13)}{\label{2}}
Let $X $ be a BK space with AK property, $Y$ be an arbitrary subset of $w$ and $R = S^{t}$, where $S=(s_{jk})$ is the inverse of the matrix $T$. Then $A \in (X_{T}, Y)$ if and only if $B^{A} \in (X, Y)$ and
$W^{A_{n}} \in (X, c_{0})$ for all $n =0, 1,2, \cdots$, where $B^A$ is the matrix with rows $B_n^A=R(A_n)$, $A_n$ are the rows of $A$ and the triangles $W^{A_n}$ for $n\in \mathbb{N}_0$ are defined by
\begin{displaymath}
  w^{A_{n}}_{pk}  = \left\{
     \begin{array}{ll}
        {  \sum\limits_{j=p}^{\infty}a_{nj}s_{jk} }& : \quad 0\leq k \leq p \\
        0              & :  \quad  k > p.
     \end{array}
   \right.
\end{displaymath}
\end{thm}
\begin{thm} {(\cite{JAR})}
Let $Y$ be any linear subspace of $w$. Then $A \in (c_{T}, Y)$ if and only if $R_k(A_{n}) \in (c_{0}, Y)$ and $W^{A_{n}} \in (c, c)$ for all $n$ and
$R_k(A_{n})e - (\gamma_{n})\in Y$, where $\gamma_{n} = \displaystyle \lim_{p \rightarrow \infty} \sum\limits_{k=0}^{p}w^{A_n}_{pk}$ for $n=0,1,2\cdots$.\\
Moreover, if $A \in (c_{T}, Y )$ then we have
$$ Az = R_k(A_{n})(T(z)) - \eta (\gamma_{n}) \quad  {\rm for~ all }~ z \in c_T, ~{\rm where}~ \eta = \displaystyle \lim_{k \rightarrow \infty}T_{k}(z).$$
\end{thm}
 To characterize the matrix transformations $A\in (X(r, s,t; \Delta^{(m)}), Y)$ for $X, Y\in \{l_{\infty}, c, c_0\}$, we list the following conditions:\\
 \begin{align}
 &\displaystyle\sup_{n}\sum\limits_{k=0}^{\infty}|R_{k}(A_{n})| < \infty\\
 &\displaystyle\lim_{n \rightarrow \infty } R_{k}(A_{n}) =0 \quad  \mbox{~for all~} k\\
&\displaystyle\sup_{p} \sum\limits_{k=0}^{p}|w_{pk}^{A_{n}}| < \infty \quad \mbox{~for all~} n \\
&\displaystyle\lim_{p \rightarrow \infty}w_{pk}^{A_{n}} =0 \mbox{~for all~} n \\
&\displaystyle\lim_{n \rightarrow \infty } R_{k}(A_{n})  \mbox{~exists for all~} k\\
&\displaystyle\lim_{n \rightarrow \infty}\sum\limits_{k=0}^{\infty}|R_{k}(A_{n})| = 0
\end{align}
\begin{align}
&\displaystyle\lim_{p \rightarrow \infty} \sum\limits_{k=0}^{p}|w_{pk}^{A_{n}}| =0  \quad \mbox{~for all~} n\\
&\displaystyle\lim_{n \rightarrow \infty } \sum_{k=0}^{\infty}\Big|R_{k}(A_{n}) - \lim_{n \rightarrow \infty }R_{k}(A_{n})\Big|=0\\
&\displaystyle\lim_{p \rightarrow \infty}w_{pk}^{A_{n}} \mbox{~exists for all~} k, n \\
&\displaystyle\lim_{p \rightarrow \infty} \sum\limits_{k=0}^{p} w_{pk}^{A_{n}} \mbox{~exists for all~} n\\
 &R_{k}(A_{n})e - (\gamma_{n}) \in c_{0}  \quad  \mbox{~for all~}  \gamma_{n}, ~n=0, 1, 2, \cdots\\
 &R_{k}(A_{n})e - (\gamma_{n}) \in l_{\infty} \quad  \mbox{~for all~} \gamma_{n},~ n=0, 1, 2, \cdots\\
 &R_{k}(A_{n})e - (\gamma_{n}) \in c  \quad  \mbox{~for all~}  \gamma_{n},~ n=0, 1, 2, \cdots,
 \end{align}
 where $\gamma_{n}= \displaystyle\lim_{p \rightarrow \infty} \sum\limits_{k=0}^{p} w_{pk}^{A_{n}}$,\\
 $R_{k}(A_{n}) = r_{k} \bigg [ \frac{a_{nk}}{s_{0}t_{k}} + \sum\limits_{i=k}^{k+1}(-1)^{i-k} \frac{D_{i-k}^{(s)}}{t_{i}} \sum\limits_{j=k+1}^{\infty}\binom{m+j-i-1} {j-i}  a_{nj} +
\sum\limits_{l=2}^{\infty} (-1)^{l} \frac{D_{l}^{(s)}}{t_{l+k}}\sum\limits_{j=k+l}^{\infty} \binom{m+j-k-l-1}{j-k-l} a_{nj} \bigg ]$
and\\
$w^{A_{n}}_{pk} =r_{k} \bigg [\sum\limits_{i=k}^{p}(-1)^{i-k} \frac{D_{i-k}^{(s)}}{t_{i}}\sum\limits_{j=p}^{\infty}\binom{m+j-i-1} {j-i} a_{nj} +
\sum\limits_{i=p+1}^{\infty}(-1)^{i-k} \frac{D_{i-k}^{(s)}}{t_{i}}\sum\limits_{j=i}^{\infty}\binom{m+j-i-1}{j-i} a_{nj} \bigg ].$

\begin{thm}
(a)  $A \in (c_{0}(r, s,t; \Delta^{(m)}), c_{0})$ if and only if $ (4.13), (4.14), (4.15)$ and $(4.16)$ hold.\\
(b) $A \in (c_{0}(r, s,t; \Delta^{(m)}), c)$ if and only if  $(4.13), (4.15), (4.16)$ and $(4.17)$hold.\\
(c) $A \in (c_{0}(r, s,t;\Delta^{(m)}), l_{\infty})$ if and only if $(4.13), (4.15)$ and $(4.16)$ hold.
\end{thm}
\begin{proof}We only prove the part (a) of this theorem. The other parts follow in a similar way. We first compute the matrices $B^A= (R_{k}(A_{n}))$ and $W^{A_{n}} =(w_{pk}^{A_{n}})$ for $n =0,1,2, \cdots$ of Theorem $\ref{2}$ to determine the conditions $B^A\in (c_0, c_0)$ and $W^{A_{n}}\in (c_0, c_0)$.
Using the same lines of proof as used in Theorem 4.8, we have
\begin{align*}
R_{k}(A_{n}) &= \sum\limits_{j=k}^{\infty}s_{jk}a_{nj}\\
& = \frac{D_{0}^{(s)}}{t_{k}} r_{k}a_{nk}  +  \sum\limits_{j =k+1}^{\infty} \sum\limits_{i=k}^{j}(-1)^{i-k} \binom{m+j-i-1} {j-i}\frac{D_{i-k}^{(s)}}{t_{i}} r_{k}a_{nj}\\
  & =   r_{k} \bigg [ \frac{a_{nk}}{s_{0}t_{k}} + \sum\limits_{i=k}^{k+1}(-1)^{i-k} \frac{D_{i-k}^{(s)}}{t_{i}} \sum\limits_{j=k+1}^{\infty}\binom{m+j-i-1} {j-i}  a_{nj} + \\
& ~~~~~~~~~~~~~~~~~~~~~~~~~~~~~~~~~~~~~~~~~\sum\limits_{l=2}^{\infty} (-1)^{l} \frac{D_{l}^{(s)}}{t_{l+k}}\sum\limits_{j=k+l}^{\infty} \binom{m+j-k-l-1} {j-k-l} a_{nj} \bigg ]
\end{align*}
and
\begin{align*}
w^{A_{n}}_{pk} &= \sum\limits_{j=p}^{\infty}s_{jk}a_{nj}\\
& = r_{k} \bigg [\sum\limits_{i=k}^{p}(-1)^{i-k} \frac{D_{i-k}^{(s)}}{t_{i}}\sum\limits_{j=p}^{\infty}\binom{m+j-i-1} {j-i} a_{nj} +
\sum\limits_{i=p+1}^{\infty}(-1)^{i-k} \frac{D_{i-k}^{(s)}}{t_{i}}\sum\limits_{j=i}^{\infty}\binom{m+j-i-1} {j-i}a_{nj} \bigg ].
\end{align*}
Using Theorem \ref{2}, we have $A\in (c_{0}(r, s,t; \Delta^{(m)}), c_{0})$ if and only if the conditions $(4.13), (4.14), (4.15)$ and $(4.16)$ hold.
\end{proof}

We can also obtain the following results.
\begin{cor}
(a)  $A \in (l_{\infty}(r, s,t; \Delta^{(m)}), c_{0})$ if and only if the conditions $(4.18)$ and $(4.19)$ hold.\\
(b) $A \in (l_{\infty}(r, s,t; \Delta^{(m)}), c)$ if and only if the conditions $(4.13), (4.17), (4.19)$ and $(4.20)$ hold.\\
(c) $A \in (l_{\infty}(r, s,t; \Delta^{(m)}), l_{\infty})$ if and only if the conditions $(4.13)$ and $(4.19)$ hold.
\end{cor}
\begin{cor}
(a)  $A \in (c(r, s,t; \Delta^{(m)}), c_{0})$ if and only if the conditions $(4.13), (4.14), (4.15), (4.21), (4.22)$ and $(4.23)$ hold.\\
(b) $A \in (c(r, s,t; \Delta^{(m)}), c)$ if and only if the conditions  $(4.13), (4.15), (4.17), (4.21), (4.22)$ and $(4.25)$ hold.\\
(c) $A \in (c(r, s,t; \Delta^{(m)}), l_{\infty})$ if and only if the conditions  $(4.13), (4.15), (4.21), (4.22)$ and $(4.24)$ hold.
\end{cor}
\section{Compact operators on the spaces $X(r, s,t; \Delta^{(m)})$ for $X  \in \{ c_0, l_{\infty} \}$}
In this section, we apply the Hausdorff measure of noncompactness to establish necessary and sufficient conditions for an infinite matrix to be a compact operator from the space $X(r, s,t; \Delta^{(m)})$ to $X$ for $X  \in \{ c_0, l_{\infty} \}$.
\par As the matrix transformations between $BK$ spaces are continuous, it is quite natural to find necessary and sufficient conditions for a matrix mapping between $BK$ spaces to be a compact operator. This can be achieved with the help of Hausdorff measure of noncompactness.
Recently several authors, namely, Malkowsky and Rako\v{c}evi\'{c} \cite{MAL5}, Dojolovi\'{c} et al. \cite{DJO2}, Dojolovi\'{c} \cite{DJO},  Mursaleen and Noman (\cite{MUR2}, \cite{MUR3}), Ba\c{s}arir and Kara \cite{BAS}
 etc. have established some identities or estimates for the operator norms and the Hausdorff measure of noncompactness of matrix operators from an
 arbitrary $BK$ space to arbitrary $BK$ space. Let us recall some definitions and well-known results.

\par Let $X$, $Y$ be two Banach spaces and $S_X$ denotes the unit sphere in $X$, i.e., $S_X=\{x\in X: \|x\|=1\}$.
 We denote by $\mathcal{B}(X, Y)$, the set of all bounded (continuous) linear operators $L: X\rightarrow Y$, which is a Banach space with the operator norm $\|L\|=\displaystyle\sup_{x\in S_X}\|L(x)\|_Y$ for all $L\in \mathcal{B}(X, Y)$. A linear operator $L: X\rightarrow Y$ is said to be compact if the domain of $L$ is all of $X$ and for every bounded sequence $(x_n)\in X$, the sequence $(L(x_n))$ has a subsequence which is convergent in $Y$ and we denote by $\mathcal{C}(X, Y)$, the class of all compact operators in $\mathcal{B}(X, Y)$. An operator $L\in \mathcal{B}(X, Y)$ is said to be finite rank if ${\rm dim}R(L)<\infty$, where $R(L)$ is the range space of $L$.
If $X$ is a $BK$ space and $a=(a_k)\in w$, then we consider
\begin{equation}{\label{eq0}}
\|a\|_X^*=\displaystyle\sup_{x\in S_X}\Big|\displaystyle\sum_{k=0}^{\infty}a_kx_k\Big|,
\end{equation}
provided the expression on the right side exists and is finite which is the case whenever $a\in X^\beta$ \cite{MUR3}.\\
Let $(X,d )$ be a metric space and $\mathcal{M}_{X}$ be the class of all bounded subsets of $X$.
Let $B(x, r) = \{y \in X : d(x,y) < r \}$ denotes the open ball of radius $r> 0$ with centre at $x$.
The Hausdorff measure of noncompactness of a set $Q \in \mathcal{M}_{X}$, denoted by $\chi(Q)$, is defined as
$$ \displaystyle\chi(Q) = \inf \Big \{ \epsilon > 0: Q \subset \bigcup_{i=0}^{n}B(x_i, r_i), x_i \in X, r_i < \epsilon , n \in \mathbb{N}_0\Big\}.$$
The function $\chi: \mathcal{M}_{X} \rightarrow [0, \infty) $ is called the Hausdorff measure of noncompactness. The basic properties of the Hausdorff measure of noncompactness can be found in (\cite{DJO1}, \cite{MAL5}, \cite{DJO2}, \cite{MAL3}, \cite{MAL4}).
For example, if $Q, Q_1$ and $Q_2$ are bounded subsets of a metric space $(X,d)$ then
\begin{align*}
&\chi(Q) =0 \mbox{~if and only if ~} Q \mbox{~is totally bounded} \mbox{~and}\\
&\mbox{if~} Q_1 \subset Q_2 \mbox{~then~} \chi(Q_1) \leq \chi(Q_2).
\end{align*}
Also if $X$ is a normed space, the function $\chi$ has some additional properties due to linear structure, namely,
\begin{align*}
&\chi(Q_1 + Q_2) \leq  \chi(Q_1) + \chi( Q_2),\\
& \chi( \alpha Q) = |\alpha| \chi(Q) ~ \mbox{for all }~ \alpha \in \mathbb{K}.
\end{align*}
Let $\phi$ denotes the set of all finite sequences, i.e., of sequences that terminate in zeros. Throughout we denote $p'$ as the conjugate of $p$ for $1\leq p<\infty$, i.e., $p{'} =\frac{ p }{p-1}$ for $p>1$ and $p'=\infty$ for $p=1$. The following known results are fundamental for our investigation.\\
\begin{lem}{\rm \cite{MUR3}}{\label{lem4}}
Let $X$ denote any of the sequence spaces $c_0 $ or $l_{\infty}$. If $A\in (X, c)$, then we have
\begin{align*}
&(i)~\alpha_k= \displaystyle\lim_{n\rightarrow\infty}{a}_{nk} \mbox{~exists for all~} k\in \mathbb{N}_{0},\\
&(ii)~\alpha=(\alpha_k)\in l_1,\\
&(iii)~\displaystyle\sup_{n}\displaystyle\sum_{k=0}^{\infty}|a_{nk}-\alpha_k|<\infty, \\
&(iv)~\displaystyle\lim_{n\rightarrow\infty}A_n(x)=\displaystyle\sum_{k=0}^{\infty}\alpha_k x_k \mbox{~for all~}x=(x_k)\in X.
\end{align*}
\end{lem}
\begin{lem}{(\rm \cite{MAL5}, Theorem 1.29)}{\label{lem5}}
Let $X$ denote any of the spaces $c_0$, $c$ or $l_\infty$. Then, we have $X^\beta=l_1$ and $\|a\|_X^*= \|a\|_{l_1}$ for all $a\in l_1$.
\end{lem}

\begin{lem}{\rm \cite{MUR3}}{\label{lem6}}
Let $X \supset \phi$ and $Y$ be $BK$ spaces. Then we have $(X, Y)\subset \mathcal{B}(X, Y)$, i.e., every matrix $A\in (X, Y)$ defines an operator $L_A\in \mathcal{B}(X, Y)$, where $L_A(x)=Ax$ for all $x\in X$.
\end{lem}

\begin{lem}{\rm \cite{DJO}}{\label{lem7}}
Let $X\supset \phi$ be a $BK$ space and $Y$ be any of the spaces $c_0$, $c$ or $l_\infty$. If $A\in (X, Y)$, then we have
$$\|L_A\|= \|A\|_{(X, l_\infty)}=\displaystyle\sup_{n}\|A_n\|_X ^{*}<\infty.$$
\end{lem}

\begin{lem}{\rm \cite{MAL5}}{\label{lem9}}
Let $Q \in \mathcal{M}_{c_0}$ and $P_l : c_0 \rightarrow c_0$ $(l \in \mathbb{N}_{0})$ be the operator defined by $ P_l(x) = (x_0, x_1, \cdots, x_l, 0, 0, \cdots)$ for all $x =(x_k) \in c_0$. Then we have
$$ \chi (Q) = \displaystyle \lim_{l \rightarrow \infty}\Big( \sup_{x \in Q}\|(I -P_l)(x)\|_{\infty} \Big),$$
where $I$ is the identity operator on $c_0$.
\end{lem}
Let $z=(z_n) \in c$. Then $z$ has a unique representation $z = \hat{\ell} e + \displaystyle \sum_{n=0}^{\infty}(z_n -\hat{\ell})e_{n}$, where $\hat{\ell} = \displaystyle \lim_{n \rightarrow \infty} z_n$. We now define the projections
$P_l$ $(l \in \mathbb{N}_{0})$ from $c$ onto the linear span of $\{e, e_0, e_1, \cdots, e_l \}$ as
$$P_l(z) = \hat{\ell} e + \displaystyle \sum_{n=0}^{l}(z_n -\hat{\ell})e_{n},$$
for all $z \in c$ and $\hat{\ell}= \displaystyle \lim_{n \rightarrow \infty} z_n$.\\
 Then the following result gives an estimate for the Hausdorff measure of noncompactness in the $BK$ space $c$.
\begin{lem}{\rm  \cite{MAL5}}{\label{lem10}}
Let $Q\in \mathcal{M}_c$ and $P_l: c\rightarrow c$ be the projector from $c$ onto the linear span of $\{e, e_{0}, e_{1}, \ldots e_l\}$. Then we have
$$\frac{1}{2} \displaystyle\lim_{l \rightarrow\infty}\Big( \displaystyle\sup_{x\in Q}\|(I-P_l)(x)\|_{\infty}\Big)\leq \chi(Q)\leq \displaystyle\lim_{l\rightarrow\infty}\Big( \displaystyle\sup_{x\in Q}\|(I-P_l)(x)\|_{\infty}\Big),$$
where $I$ is the identity operator on $c$.
\end{lem}

\begin{lem}{\rm  \cite{MAL5}}{\label{lem8}}
Let $X,Y$ be two Banach spaces and $L \in \mathcal{B}(X, Y)$. Then $$\|L\|_{\chi} = \chi(L(S_X))$$
and $$L \in \mathcal{C}(X, Y) ~\mbox{if and only if}~ \|L\|_{\chi} =0.$$
\end{lem}

We establish the following lemmas which are required to characterize the classes of compact operators with the help of Hausdorff measure of noncompactness.
\begin{lem}{\label{lem11}}
Let $X(r, s, t; \Delta^{(m)})$ be any sequence spaces for $X\in\{c_0, l_{\infty}\}$. If $a=(a_k)\in [X(r, s, t; \Delta^{(m)})]^ \beta$ then
$\tilde{a}= (\tilde{a}_k)\in X^{\beta} = l_1$ and the equality $$\displaystyle\sum_{k=0}^{\infty}a_k x_k=\displaystyle\sum_{k=0}^{\infty}\tilde{a}_k y_k$$ holds for every $x=(x_k)\in X(r, s, t; \Delta^{(m)})$ and $y =(y_k) \in X$, where $y = (A(r,s,t). \Delta^{(m)})x$. In addition
\begin{equation}{\label{eq1}}
 \tilde{a}_{k}= r_{k} \bigg [ \frac{a_{k}}{s_{0}t_{k}} + \sum\limits_{i=k}^{k+1}(-1)^{i-k} \frac{D_{i-k}^{(s)}}{t_{i}} \sum\limits_{j=k+1}^{\infty}\binom{m+j-i-1} {j-i}  a_{j}  +
\sum\limits_{l=2}^{\infty} (-1)^{l} \frac{D_{l}^{(s)}}{t_{l+k}}\sum\limits_{j=k+l}^{\infty} \binom{m+j-k-l-1} {j-k-l} a_{j} \bigg ].
\end{equation}
\end{lem}
\begin{proof}
Let $a=(a_k)\in [X(r, s, t; \Delta^{(m)})]^ \beta$. Then by Theorem 4.7 and Remark 4.2, we have $R(a)= (R_{k}(a)) \in X^{\beta} =l_1$ and also
$$\displaystyle\sum_{k=0}^{\infty}a_k x_k=\displaystyle\sum_{k=0}^{\infty}R_{k}(a) T_k(x) \quad \forall~x \in X(r, s, t; \Delta^{(m)}),$$
where
\begin{center}
$R_{k}(a) = r_{k} \bigg [ \frac{a_{k}}{s_{0}t_{k}} + \sum\limits_{i=k}^{k+1}(-1)^{i-k} \frac{D_{i-k}^{(s)}}{t_{i}} \sum\limits_{j=k+1}^{\infty}\binom{m+j-i-1} {j-i}  a_{j}  +
\sum\limits_{l=2}^{\infty} (-1)^{l} \frac{D_{l}^{(s)}}{t_{l+k}}\sum\limits_{j=k+l}^{\infty} \binom{m+j-k-l-1} {j-k-l} a_{j} \bigg ]= \tilde{a}_k,$
\end{center}
and $ y =T(x)= (A(r,s,t). \Delta^{(m)})x$. This completes the proof.
\end{proof}

\begin{lem}{\label{lem12}}
Let $X(r, s, t; \Delta^{(m)})$ be any sequence spaces for $X\in\{c_0, l_{\infty}\}$. Then we have
$$\|a\|_{X(r, s, t; \Delta^{(m)})}^*= \|\tilde{a}\|_{l_1}=\displaystyle\sum_{k=0}^{\infty}|\tilde{a_k}|<\infty $$
for all $a=(a_k)\in [X(r, s, t; \Delta^{(m)})]^\beta$, where $\tilde{a}=(\tilde{a}_{k})$ is defined in (\ref{eq1}).
\end{lem}
\begin{proof}
Let $a=(a_k)\in [X(r, s, t; \Delta^{(m)})]^\beta$. Then from Lemma \ref{lem11}, we have $\tilde{a}=(\tilde{a}_{k})\in l_1$.
Also $x\in S_{X(r, s, t; \Delta^{(m)})}$ if and only if $y=T(x) \in S_X$ as $\|x\|_{X(r, s, t; \Delta^{(m)})}= \|y\|_{\infty}$. From (\ref{eq0}), we have
$$\|a\|_{X(r, s, t; \Delta^{(m)})}^*= \displaystyle\sup_{x\in S_{X(r, s, t; \Delta^{(m)})}}\Big|\displaystyle\sum_{k=0}^{\infty}a_k x_k\Big|=\sup_{y\in S_X}\Big|\displaystyle\sum_{k=0}^{\infty}\tilde{a}_{k} y_k\Big|=  \|\tilde{a}\|_X^*.$$
Using by Lemma \ref{lem5}, we have $\|a\|_{X(r, s, t; \Delta^{(m)})}^*=\|\tilde{a}\|_X^*= \|\tilde{a}\|_{l_1}$, which is finite as
$\tilde{a}\in l_1$. This completes the proof.
\end{proof}

\begin{lem}{\label{lem13}}
Let $X(r, s, t; \Delta^{(m)})$ be any sequence spaces for $X \in \{c_0, l_{\infty}\}$, $Y$ be any sequence space and $A=(a_{nk})_{n,k}$ be an infinite matrix. If $A\in (X(r, s, t; \Delta^{(m)}), Y)$ then $\tilde{A} \in (X, Y)$ such that $Ax= \tilde{A}y$ for all $x\in X(r, s, t; \Delta^{(m)})$ and $y\in X$, which are connected by the relation $y = (A(r,s,t). \Delta^{(m)})x$ and \\ $\tilde{A}=(\tilde{a}_{nk})_{n,k}$ is given by
\begin{equation}\label{eq111}
\tilde{a}_{nk}= r_{k} \bigg [ \frac{a_{nk}}{s_{0}t_{k}} + \sum\limits_{i=k}^{k+1}(-1)^{i-k} \frac{D_{i-k}^{(s)}}{t_{i}} \sum\limits_{j=k+1}^{\infty}\binom{m+j-i-1} {j-i}  a_{nj}  +
\sum\limits_{l=2}^{\infty} (-1)^{l} \frac{D_{l}^{(s)}}{t_{l+k}}\sum\limits_{j=k+l}^{\infty} \binom{m+j-k-l-1} {j-k-l} a_{nj} \bigg ],
\end{equation}
provided the series on the right side converges for all $n, k$.
\end{lem}
\begin{proof}
We assume that $A\in (X(r, s, t; \Delta^{(m)}), Y)$, then $A_n \in [X(r, s, t; \Delta^{(m)})]^{\beta}$ for all $n$. Thus it follows from Lemma \ref{lem11}, we have $\tilde{A}_n \in X^{\beta}= l_1$ for all $n$ and $Ax = \tilde{A}y$ holds for every $x \in X(r, s, t; \Delta^{(m)})$, $y \in X$, which are connected by the relation $y = (A(r, s, t). \Delta^{(m)})x$. Hence $\tilde{A}y \in Y$. Since $x = (\Delta^{(m)})^{-1} (A(r, s, t))^{-1}y$, for every $y \in X$, we get some $x \in X(r, s, t; \Delta^{(m)})$ and hence $\tilde{A} \in (X, Y)$. This completes the proof.
\end{proof}
\begin{lem}{\label{lem14}}
Let $X(r, s, t; \Delta^{(m)})$ be any sequence spaces for $X\in\{c_0, l_{\infty}\}$, $A=(a_{nk})_{n,k}$ be an infinite matrix and $\tilde{A}= (\tilde{a}_{nk})_{n,k}$ be the associate matrix defined in (\ref{eq111}). If $A \in (X(r, s, t; \Delta^{(m)}), Y)$, where $Y \in \{c_0, c, l_{\infty}\}$, then $$ \|L_A\|= \|A\|_{(X, l_\infty)}=\displaystyle\sup_{n} \sum_{k=0}^{\infty}|\tilde{a}_{nk}|<\infty.$$
\end{lem}
\begin{proof}
Since the spaces $X(r, s, t; \Delta^{(m)})$ for $X\in\{c_0, l_{\infty}\}$ are $BK$ spaces, using Lemma \ref{lem7} we have
$$ \|L_A\|= \|A\|_{(X, l_\infty)}= \displaystyle\sup_{n}\| A_n \|^{*}_{X(r, s, t; \Delta^{(m)})}.$$ Now from Lemma \ref{lem12}, we have
$$\|A_n \|_{X(r, s, t; \Delta^{(m)})}^*= \|\tilde{A_n}\|_{l_1}= \displaystyle\sum_{k=0}^{\infty}|\tilde{a}_{nk}|,$$
which is finite as $(\tilde{A_n}) \in l_1$. This completes the proof.
\end{proof}

Now we give the main results.
\begin{thm}{\label{thm1}}
Let $X(r, s, t; \Delta^{(m)})$ be any sequence spaces, where $X \in \{ c_0 , l_{\infty} \}$. \\
$(a)$ If $A\in (X(r, s, t; \Delta^{(m)}), c_0)$ then
\begin{equation}{\label{eq2}}
 \|L_A\|_\chi= \displaystyle\limsup_{n\rightarrow\infty}\displaystyle\sum_{k=0}^{\infty}|\tilde{a}_{nk}|
\end{equation}
$(b)$ If $A\in (X(r, s, t; \Delta^{(m)}), c)$ then
\begin{equation}{\label{eq3}}
\frac{1}{2} \displaystyle\limsup_{n\rightarrow\infty}\displaystyle\sum_{k=0}^{\infty}|\tilde{a}_{nk}-\tilde{\alpha}_k| \leq \displaystyle \|L_A\|_\chi \leq \displaystyle\limsup_{n\rightarrow\infty}\displaystyle\sum_{k=0}^{\infty}|\tilde{a}_{nk}-\tilde{\alpha}_k|,
\end{equation}
where $\tilde{\alpha}_k= \displaystyle\lim_{n\rightarrow\infty}\tilde{a}_{nk}$ for all $k$.\\
$(c)$ If $A \in (X(r, s, t; \Delta^{(m)}), l_{\infty})$ then
\begin{equation}{\label{eq4}}
 0 \leq \|L_A \|_{\chi} \leq \displaystyle \limsup_{n \rightarrow \infty} \sum_{k=0}^{\infty}|\tilde{a}_{nk} |.
\end{equation}
\end{thm}
\begin{proof}
$(a)$ Let us first observe that the expressions in (\ref{eq2}) and in (\ref{eq4}) exist by Lemma \ref{lem14}. Also by using the Lemma \ref{lem13} \&
\ref{lem4}, we can deduce that the expressions in (\ref{eq3}) exists.\\
We write $S= S_{X(r, s, t; \Delta^{(m)})}$ in short. Then by Lemma \ref{lem8}, we have $ \|L_A \|_{\chi} = \chi(AS)$. Since $X(r, s, t; \Delta^{(m)})$ and $c_0$ are BK spaces,
$A $ induces a continuous map $L_{A}$ from $X(r, s, t; \Delta^{(m)})$ to $c_0$ by Lemma \ref{lem6}. Thus $AS$ is bounded in $c_0$,
i.e., $AS \in \mathcal{M}_{c_0}$. Now by Lemma \ref{lem9},
$$ \chi(AS) = \displaystyle \lim_{l \rightarrow \infty}\Big( \sup_{x \in S}\|(I -P_l)(Ax)\|_{\infty} \Big), $$
where the projection $P_l : c_0 \rightarrow c_0$ is defined by $ P_l(x) = (x_0, x_1, \cdots, x_l, 0, 0, \cdots)$ for all $x =(x_k) \in c_0$ and
$l \in \mathbb{N}_{0}$. Therefore $\|(I -P_l)(Ax)\|_{\infty} = \displaystyle \sup_{n > l}|A_n(x)|$ for all $ x \in X(r, s, t; \Delta^{(m)})$.
Using (\ref{eq0}) and Lemma \ref{lem12}, we have
\begin{align*}
\displaystyle \sup_{x \in S}\|(I -P_l)(Ax)\|_{\infty} & =  \sup_{ n> l}\|A_n \|_{X(r, s, t; \Delta^{(m)})}^{*}\\
& = \sup_{ n> l}\|\tilde{A_n }\|_{l_1}
\end{align*}
Therefore $\chi(AS) = \displaystyle \lim_{l \rightarrow \infty}\Big( \sup_{ n> l}\|\tilde{A_n }\|_{l_1} \Big) = \displaystyle \limsup_{n \rightarrow \infty} \|\tilde{A_n }\|_{l_1} = \displaystyle \limsup_{n \rightarrow \infty} \sum_{k=0}^{\infty}|\tilde{a}_{nk} | $.
This completes the proof.

(b) We have $AS \in \mathcal{M}_{c}$. Let $P_l : c \rightarrow c$ be the projection from $c$ onto
the span of $\{e, e_0, e_1, \cdots, e_l \}$ defined as
$$ P_l(z) = \hat{\ell} e + \sum_{k=0}^{r} (z_k -\hat{\ell})e_k,$$
where $\hat{\ell} = \displaystyle\lim_{k \rightarrow \infty} z_k$.
Thus for every $l \in \mathbb{N}_{0}$, we have $$ (I -P_l)(z) = \sum_{k=l+1}^{\infty} (z_k - \hat{\ell})e_k.$$
Therefore
$ \|(I -P_l)(z)\|_{\infty} = \displaystyle \sup_{k > l}|z_k - \hat{\ell}| $ for all $z=(z_k) \in c$.
Applying Lemma \ref{lem10}, we have
\begin{equation}{\label{eq10}}
\frac{1}{2} \displaystyle\lim_{l\rightarrow\infty}\Big( \displaystyle\sup_{x\in S}\|(I-P_l)(Ax)\|_{\infty}\Big)\leq \| L_A\|_{\chi} \leq \displaystyle\lim_{l\rightarrow\infty}\Big( \displaystyle\sup_{x\in S}\|(I-P_l)(Ax)\|_{\infty}\Big).
\end{equation}
Since $A \in (X(r, s, t; \Delta^{(m)}), c),$ we have by Lemma \ref{lem13}, $\tilde{A }\in (X, c)$ and $Ax = \tilde{A }y$ for every $x \in X(r, s, t; \Delta^{(m)})$ and $y \in X$, which are connected by the relation $y =(A(r,s,t). \Delta^{(m)})x$.
Using Lemma \ref{lem4}, we have $\tilde{\alpha}_k = \displaystyle \lim_{n \rightarrow \infty}\tilde{a}_{nk}$ exists for all $k$,
$ \tilde{\alpha} = (\tilde{\alpha}_k ) \in X ^{\beta} = l_1$ and $\displaystyle \lim_{n \rightarrow \infty}\tilde{A}_{n}(y) = \sum_{k=0}^{\infty} \tilde{\alpha}_k y_k.$
Since $ \|(I -P_l)(z)\|_{\infty} = \displaystyle \sup_{k > l}|z_k - \hat{\ell}| $, we have
\begin{align*}
\|(I-P_l)(Ax)\|_{\infty} & = \|(I-P_l)(\tilde{A}y)\|_{\infty}\\
& = \sup_{ n> l}\Big | \tilde{A}_n(y) - \sum_{k=0}^{\infty} \tilde{\alpha}_k y_k \Big |\\
& =  \sup_{ n> l}\Big | \sum_{k=0}^{\infty}(\tilde{a}_{nk} -  \tilde{\alpha}_k) y_k \Big |.
\end{align*}
Also we know that $x \in S= S_{X(r, s, t; \Delta^{(m)})}$ if and only if $y \in S_X$. From (\ref{eq0}) and Lemma \ref{lem5}, we deduce
\begin{align*}
\sup_{x \in S} \|(I-P_l)(Ax)\|_{\infty} & =  \sup_{ n> l}\Big(\sup_{y \in S_X}\Big | \sum_{k=0}^{\infty}(\tilde{a}_{nk} -  \tilde{\alpha}_k) y_k \Big | \Big)\\
& = \sup_{ n> l}\| \tilde{A}_n -\tilde{\alpha }\|_{X}^{*} =  \sup_{ n> l}\| \tilde{A}_n -\tilde{\alpha }\|_{l_1}.
\end{align*}
Hence from (\ref{eq10}), we have
\begin{center}
 $\frac{1}{2} \displaystyle\limsup_{n\rightarrow\infty}\displaystyle\sum_{k=0}^{\infty}|\tilde{a}_{nk}-\tilde{\alpha}_k| \leq \displaystyle \|L_A\|_\chi \leq \displaystyle\limsup_{n\rightarrow\infty}\displaystyle\sum_{k=0}^{\infty}|\tilde{a}_{nk}-\tilde{\alpha}_k|$.
\end{center}

$(c)$ We first define a projection $P_l : l_{\infty} \rightarrow l_{\infty}$, as $ P_l(x) = (x_0, x_1, \cdots, x_l, 0, 0, \cdots)$ for all $x =(x_k) \in l_{\infty}$, $l \in \mathbb{N}_{0}$. We have
$$  AS \subset P_l(AS) + (I -P_l)(AS).$$
By the property of $\chi$, we have
\begin{align*}
 0 \leq \chi(AS)& \leq \chi(P_{l}(AS)) + \chi((I - P_{l})(AS))\\
 & = \chi((I - P_{l})(AS))\\
 & \leq \sup_{x \in S} \|(I-P_l)(Ax)\|_{\infty} \\
 &= \sup_{n> l} \| \tilde{A}_n\|_{l_1}.
 \end{align*}
 Hence
\begin{center}
 $0 \leq \chi(AS) \leq \displaystyle \limsup_{n \rightarrow \infty} \| \tilde{A}_n\|_{l_1} = \displaystyle\limsup_{n \rightarrow \infty}\displaystyle\sum_{k=0}^{\infty}|\tilde{a}_{nk}|.$
\end{center}
This completes the proof.
\end{proof}
\begin{cor}{\label{cor1}}
Let $X(r, s, t; \Delta^{(m)})$ be any sequence spaces for $X \in \{ c_0, l_{\infty}\}$. \\
$(a)$ If $A \in (X(r, s, t; \Delta^{(m)}), c_0)$, then $L_A$ is compact if and only if
$\displaystyle \lim_{n \rightarrow \infty}\sum_{k=0}^{\infty}|\tilde{a}_{nk}| =0$\\
$(b)$ If $A \in (X(r, s, t, \Delta^{(m)}), c)$ then
\begin{center}
$L_A$ is compact if and only if
$\displaystyle \lim_{n \rightarrow \infty} \sum_{k=0}^{\infty}|\tilde{a}_{nk} -\tilde{\alpha}_k | =0$,
where ${\tilde{\alpha}}_k = \displaystyle \lim_{n \rightarrow \infty}\tilde{a}_{nk}$ for all $k $.
\end{center}
$(c)$ If $A \in (X(r, s, t, \Delta^{(m)}), l_{\infty})$ then $L_A$ is compact if and only if
$\displaystyle \lim_{n \rightarrow \infty} \sum_{k=0}^{\infty}|\tilde{a}_{nk}| =0$.
\end{cor}

\begin{proof}
 The proof is immediate from the Theorem \ref{thm1}.
\end{proof}
\begin{cor}
For every matrix $A \in (l_{\infty}(r, s, t; \Delta^{(m)}), c_0)$ or $A \in (l_{\infty}(r, s, t; \Delta^{(m)}), c)$ the operator $L_{A}$, induces by matrix $A$ is compact.
\end{cor}
\begin{proof}
Let $A \in (l_{\infty}(r, s, t; \Delta^{(m)}), c_0)$ then $\tilde{A }\in (l_{\infty}, c_0)$, where $Ax = \tilde{A}y$ holds for every
$x \in l_{\infty}(r, s, t; \Delta^{(m)})$ and $y \in l_{\infty}$, which are connected by the relation $y =(A(r,s,t). \Delta^{(m)})x$.
Since $\tilde{A }\in (l_{\infty}, c_0)$, by Theorem 4.4(d), we have $ \displaystyle\lim_{n \rightarrow \infty}\sum_{k=0}^{\infty}|\tilde{a}_{nk}| =0$. Hence by Corollary \ref{cor1}(a) the operator $L_A$ is compact.\\
 Similarly if $A \in (l_{\infty}(r, s, t; \Delta^{(m)}), c)$ then $\tilde{A }\in (l_{\infty}, c)$. From Theorem 4.4(g), we have $ \displaystyle\lim_{n \rightarrow \infty}\sum_{k=0}^{\infty}|\tilde{a}_{nk} - \tilde{\alpha}_k|  =0$, where
$\tilde{\alpha}_k = \displaystyle \lim_{n \rightarrow \infty}\tilde{a}_{nk}$ for all $k$. Thus by Corollary \ref{cor1}(b), we have $L_A$ is compact.
\end{proof}

\end{document}